\documentclass[12pt]{amsart}

\usepackage{amsthm,amsmath,amssymb,mathcomp}
\numberwithin{equation}{section}

\def\ca{{\mathcal A}}

\def\cc{{\mathcal C}}

\def\bc{{\mathbb C}}

\def\bt{{\mathbb T}}
\def\bz{{\mathbb Z}}

\def\a{\alpha}
\def\eps{\varepsilon}

\def\d{\delta}

\def\g{\gamma}

\theoremstyle{plain}
\newtheorem{lemma}{Lemma}[section]
\newtheorem{proposition}[lemma]{Proposition}
\newtheorem{theorem}[lemma]{Theorem}

\newtheorem{lem}{Lemma}[section]

\theoremstyle{definition}

\newtheorem{definition}[lemma]{Definition}

\newtheorem*{definition*}{Definition}
\newtheorem{defn}[lem]{Definition}

\theoremstyle{remark}

\begin{document}

\title[Hadamard and block  Schur products ] {\textsc{ The  Hadamard  Product in  a Crossed Product C*-algebra  }  }

\author{Erik Christensen}
\address{\hskip-\parindent
Erik Christensen, Institut for Mathematiske Fag, University of Copenhagen,
Copenhagen, Denmark.}
\email{echris@math.ku.dk}

\begin{abstract}
Given two $n \times n $  matrices $A = (a_{ij})$ and $B=(b_{ij}) $ with entries in $B(H)$ for some Hilbert space $H,$ their {\em block Schur product} is the  $n \times n$ matrix $ A\square B := (a_{ij}b_{ij}).$ Given two continuous functions $f,g$ on the torus with Fourier coefficients $(f_n), \, (g_n)$ their convolution product $f \star g$  has Fourier coefficients $(f_ng_n).$ Based on this, the  {\em  Schur product }  on scalar matrices  is also known as the {\em  Hadamard product } 

We show that for a C*-algebra $\ca,$ and  a discrete group $G $ with an action $\a_g$ of $G$ on $\ca$ by *-automorphisms, the reduced  crossed product C*-algebra C$^*_r(\ca, \a, G)$ possesses a natural generalization of the convolution product, which we suggest should be named  {\em the  Hadamard product. } 

We show that this product has a natural Stinespring representation and we lift some known results on block Schur products to this setting, but we also show that the  block Schur product is a special case of the Hadamard product in a crossed product algebra.     
\end{abstract}

\date{\today}
\maketitle

\section{Introduction}

Based on an interest in properties of  spectral triples, which is a basic object in Connes' { \em noncommutative geometry, } we have been studying the block Schur product - the entry wise product  on infinite matrices over $B(H)$ - for some time \cite{C0, C1, C2}. The leading expert in this field, Roger A. Horn has always used the term {\em Hadamard  product } and he explains in \cite{Ho} p. 92 - 93 and \cite{HJ} p. 302 - 303 the reason why. I will quote Halmos' answer to Horn, when Horn asked Halmos,  why he used the term  {\em Hadamard product. } Halmos answered {\em because Johnny said so.} To  those who may not know, I can tell, that P. R.  Halmos was John von Neumann's assistant in Princeton. 
This is of course a perfectly good reason, but for a person coming from outside, this seemed a bit unfair towards Schur, who actually studied this product in details  and published the  fundamental article on  this product on bilinear forms in \cite{Sc}. Hadamard's product \cite{Ha}   is based on a product of functions represented by Laurent series, and it is obtained  via element-wise-products of the coefficients. Schur's product is really based on a study of  properties of matrices, so I found that in my studies, Schur's approach seemed to fit me best.
 
  The present article studies the  reduced crossed product $\cc := $C$^*_r(\ca, \a, G)$  of a C*-algebra $\ca$ by an action $\a$  of a discrete group $G,$ a construction you may find described in details in  \cite{KR} Definition 13.1.1. In this version of the  reduced crossed product, the C*-algebra $\ca$ acts on a Hilbert space $H,$  there is an injective  *-representation $\Psi$ of $\ca$ on the Hilbert space $H \otimes \ell_2(G)$  and a unitary representation $g \to L_g$ of $G$ on the same Hilbert space. The reduced crossed product $\cc$  is the C*-algebra generated by all the operators $\Psi(A)$ and $L_g.$ There exists a faithful conditional expectation $\pi$  of $\cc$ onto $\Psi(\ca)$ which satisfies $\pi(L_g\Psi(A)) = \d(e,g)\Psi(a).$ The mapping $\pi$ makes it possible  to describe elements in $\cc$ via generalized Fourier series over $G$ with coefficients in $\Psi(\ca)$  such as $$\cc \ni X:  X  \sim \sum_{g \in G}L_g X_g \text{ and } X_g := \pi(L_g^*X), \, X_g \in \Psi(\ca).$$ It is known, that this sum is convergent in a topology, which depends on $\pi,$ and we will show in Section 2 that this convergence result may be viewed as a direct generalization of the classical Plancherel theorem.   We study a product $\star$ on $\cc$ which is usually called the convolution product and it is defined at the level of the generalized Fourier series via the formula
  $$ \forall X \sim \sum_{g \in G } L_gX_g, \, Y \sim \sum_{g \in G} L_gY_g: \quad X \star Y \sim \sum_{g \in G} L_g X_gY_g.$$ In this article we show that this product has a natural  {\em Stinespring representation } as a {\em completely bounded bilinear operator } in the sense of \cite{CS}.  On the other hand the product is a straight forward generalization of the product Hadamard studied in \cite{Ha}. Since  the convolution operation is defined via an integral over the dual group, the  convolution   does not always exist in the setting of a general discrete group, so we think it is reasonable to call this product the Hadamard product. 
 Later on we show that the block Schur product is a special case of this product.  In the end, it turns out that  the results we have obtained for block Schur products during the last couple of years \cite{C1, C2} do extend to this Hadamard product in a reduced crossed product C*-algebra.

As mentioned above, we   
 show that the Hadamard product  - seen as a bilinear operator on $\cc \times \cc $ - has a natural Stinespring representation $$ X \star Y \, = \, V^*\rho(X) F \rho(Y) V, $$ where $V$ is an isometry, $\rho$ a representation of $\cc$  and $F$  a self-adjoint unitary. It is worth to remark that the conditional expectaion $\pi$ of $\cc$ onto $\Psi(\ca)$ is given via the same  representation $\rho$ and the same isometry $V$ as $$\pi(X) \, = \, V^*\rho(X) V.$$ 
  This decomposition of the  bilinear operator $\star$ has some similarities with the description of a bilinear form on $\bc^n$ via a scalar matrix, and in this analogy we can see that the Hadamard product is a symmetric operator. It is not positive, since $F$ is  a non trivial  self-adjoint unitary, and  then we get
 
  $$  \forall X \in \cc: \quad  -\pi(X^*X)\leq X^* \star X \leq  \pi(X^*X)$$
 Since the Stinespring decompositions of both $\star$ and the conditional expectation $\pi$ are based on the same representation $\rho $ and the same outer multipliers $V^*$ and $V,$ it is easy to obtain the following identity
 \begin{align*}  &\forall X, Y  \in \cc \, \exists S(X,Y) \in B\big( H \otimes \ell_2(G)\big), \,\|S(X,Y)\| \leq 1: \\   & X \star Y\, = \pi(XX^*)^{(1/2)} S(X,Y) \pi(Y^*Y)^{(1/2)},
 \end{align*}
 and from here follows immediately that Livshits' inequality \cite{Li} extends to the Hadamard product as 
 $$  \forall X, Y  \in \cc: \quad  \|X \star Y\|\, \leq  \|\pi(XX^*)^{(1/2)}\|\|  \pi(Y^*Y)^{(1/2)}\|,$$
 and this is a generalization of the classical inequality from Fourier analysis
 $$\| g \star h\|_\infty \, \leq \, \|g\|_2 \|h\|_2.$$

In the proofs we use some existing theory on the reduced crosssed product of a C*-algebra by a discrete group, and in section 2 we present the rsults we need in a way which demonstrates that they actually form a quite natural extension of Plancherel's theorem.

\section{Extension of the Plancherel Theorem} 

The content of this section is not new and it is closely connected to the fundamental work \cite{Ze} by Zeller-Meier on crossed products. We are in need of the results  presented here to show our main results in Section 3, and in the process of writing this down, we realized that the statements may be known by many people, so with the help of Erik B{\'e}dos, we found the needed results  in the articles \cite{BC1, BC2} by B{\'e}dos and Conti and the article \cite{RS} by R{\o}rdam and Sierarkowski. Nevertheless we have included our own proofs of the propositions \ref{norm} and \ref{normconv} because we think that they illustrate our point of view, namely that the results form a clear extension of the classical results named {\em Parceval's theorem } and {\em  Plancherel's theorem.} We think that this aspect is most easily demonstrated, when the generalized Fourier series are written in the form $ X \sim \sum_g  L_gX_g$ rather than $X \sim \sum_gX_gL_g,$ which is used in most articles. This is nothing but a notational difference, which can not be detected in classical Fourier analysis.   

We will consider a C*-algebra $\ca$ acting on a Hilbert space $H$ and an action, by *-automorphisms, $\a_g$ on $\ca $ by  a discrete group $G.$ We will use the notation of \cite{KR} Definition 13.1.1 and define the  reduced left crossed product $$\cc \,:= \, \mathrm{C}^*_r(\ca, \a_g, G)$$ as the C*-algebra generated by the operators on $H \otimes \ell_2(G)$ given by $$ \Psi(A) := \sum_{g \in G} \a_g^{-1}(A) \otimes E_g, A \in \ca, \,  L_g := I \otimes l_g, \, g \in G. $$ It is well known that there is a faithful conditional expectation $\pi$ of $\cc$ onto $\Psi(\ca)$ which satisfies $$\pi(L_g\Psi(A)) \, = \, \delta(g,e) \Psi(A).$$ It is worth to remark that by \cite{BC1} Proposition 3.1, which is based on \cite{Ze}, any faithful covariant representation of the C*-dynamical system $(\ca, \a_g, G)$  which has a faithful conditional expectation from the crossed product C*-algebra onto its copy of $\ca $ will be an isomorphic copy of the reduced crossed product we look at here. Since our starting point is the concretely given C*-algebra  $\ca$ acting on  $H$ we are studying the left regular reduced crossed product, but we have chosen a specific representation of it.  Before we present the results  we will like to mention the text books \cite{Pe, Wi} by G. K. Pedersen and D. Williams, respectively,  on crossed product C*-algebras.

   In \cite{KR} the authors  study von Neumann algebras, whereas we  study  C*-algebras, and we want to emphasize that there is a difference. In particular the proposition \ref{normconv} does not hold in a von Neumann algebraic setting, and you may look in \cite{Me} by Mercer, who presents examples of {\em non convergence } in a   crossed product of a von Neumann algebras by a discrete group. 

The classical Fourier theory is a special case of the left regular crossed product of a discrete  C*-dynamical system. In that case the algebra $\ca$ is just the scalars $\bc,$ the group is $\bz$ and the action is trivial. The Hilbert space is $L^2(\bt)$ and the unitaries $L_n$ are the multiplication operators $M_{z^n} $ which multiply an $L^2(\bt)$ function by $z^n.$ The C*-algebra $\cc$ then consists of the multiplication operators $M_f,$ where $f$ is a continuous function on $\bt.$  The conditional expectation $\pi$ is a state in this case. It is given by  $$\pi(M_f) \, = \, \frac{1}{2\pi}\int_{- \pi}^{\pi}f(e^{i\theta})d\theta,$$ and  we find that the $n$'th
Fourier coefficient of $f$ is given by $ f_n = \pi(L_n^*M_f),$
so we will use the following terminology.

\begin{definition}
For an $X$ in $\cc$  and a $g$ in $G$  we define the coefficient $X_g$ of $ X$ as $X_g := \pi (L_g^*X).$ 
\end{definition}
Staying in the case of the continuous functions on the circle, but focussing on the conditional expectation $\pi,$ we can reformulate Parceval's   identity for Fourier series of square integrable $2\pi-$periodic functions   to  obtain \begin{align} &\label{Plancherel}  \text{ \bf Plancherel's Thm.} \quad \pi(M_f^*M_f) \, = \, \|f\|_2^2 \, = \, \sum_{n \in \bz} |f_n|^2 \, = \, \sum_{ n \in \bz}\pi(L_n^*M_f)^*\pi(L_n^*M_f),\\
&\label{Nearest}  \text{ \bf Nearest point Thm.}\quad \forall K \subseteq \bz\, \forall g \in \mathrm{span}(\{ z^k\, : \, k \in K\,\}):  \\ & \pi\big((M_f - \sum_{h \in K}f_hM_{z^h})^*(M_f - \sum_{k \in K}f_kM_{z^k})\big) \, \leq \,  \pi\big((M_f - M_g)^*(M_f - M_g)\big) .\notag  
\end{align}
and we will show that these formulas hold in our setting too. 
We start by defining the analogue to the square integrable norm.

\begin{proposition} \label{norm} The expression  $$ \forall X \in \cc: \, \,\|X\|_{\pi} := \|\pi(X^*X)\|^{(1/2)} $$ defines a norm on $\cc.$  
\end{proposition}
\begin{proof}
By Stinespring's theorem \cite{St} there exists a Hilbert space $K$ a representation $\rho$ of $\cc$ on $K$ and a contraction $V$ in $B(H \otimes \ell_2(G),K)$ such that for any $X$ in $\cc$ we have $\pi(X) \, = \, V^*\rho(X) V.$ From here we get via the C*-algebraic norm identity, $\|Z^*Z\| = \|Z\|^2$ that $$\|X\|_{\pi} \, = \, \|(V^*\rho(X)^*\rho(X) V\|^{(1/2)} \, = \, \|\rho(X)V\|.$$ 
With this identity in mind it follows easily that $\|.\|_{\pi}$ is a seminorm and since $\pi$ is faithful in this construction it follows that $\|.\|_{\pi}$ is a norm. 
\end{proof}
We can now extend Plancherel's theorem and the nearest point result.
\begin{proposition} \label{normconv}
Let $X$ be in $\cc,$  $K $ a finite subset of $G$ and $(Y_k)_{k \in K}$ elements in $\Psi(\ca)$  then    $$ \pi\big((X - \sum_{h \in K} L_hY_h )^*(X - \sum_{k \in K} L_kY_k )\big) \geq \pi\big((X - \sum_{h \in K} L_hX_h )^*(X - \sum_{k \in K} L_kX_k )\big) . $$
The series $\sum_{g \in G}L_gX_g$ converges in the norm $\|.\|_\pi$ to $X.$ The series $\sum_{g \in G } X_g^*X_g $  converges in norm to $\pi(X^*X)$ and for  $\hat X : = \sum_{g \notin K}L_gX_g$ we have $$ \|\hat X\|_\pi\, = \,\|\sum_{g \notin K} X_g^*X_g \|^{(1/2)} \, =  \,\underset{ Y \in \mathrm{span}(\{L_k\Psi(A_k)\, :\, k \in K A_k \in \ca\,\})}{\inf}\|X - Y\|_\pi. $$  
\end{proposition}
 \begin{proof}
  Remember that $\pi$ is an $\Psi(\ca)-$bimodular map, so we have 
 \begin{align} \label{positive}
 & \, \pi\big((X - \sum_{h \in K} L_hY_h)^*(X - \sum_{k \in K}L_kY_k)\big)\notag \\  = \,& \pi(X^*X) - \sum_{k \in K} X_k^*Y_k - \sum_{h\in K}Y_h^* X_h + \sum_{h \in K} Y_h^*Y_h \notag\\  
 =\,& \pi(X^*X) - \sum_{g \in K}X_g^*X_g. + \sum_{k \in K}(X_k - Y_k)^*(X_k - Y_k) \notag \\ 
\geq  \,& \pi(X^*X) - \sum_{g \in K} X^*_gX_g \notag \\ = \, & \pi\big( (X - \sum_{h \in K}L_hX_h)^*(X - \sum_{k \in K }L_kX_k)\big)\notag \\ = \,& \pi(\hat X^* \hat X) \, \geq \, 0.
\end{align}  
where the second last equality sign again is based on the $\Psi(\ca)$ modularity of $\pi.$ 
By the definition of $\cc$ there exists to any $\eps > 0$ a finite subset $L$ of $G$ and for any $g$ in $L$ an operator $Y_g$ in $\Psi(\ca)$ such that $\|X - \sum_{g \in L}L_g Y_g\| < \eps,$ and we may with the help of the inequality above estimate that for any finite subset $K$ containing $L$ 

\begin{align}
\eps^2I_{B\big(H\otimes \ell_2(G)\big)}\,& \geq \, \pi\big(( X - \sum_{h \in L} L_hY_h)^*(X - \sum_{k \in L} L_kY_k)\big)\notag \\ &\geq \,  \pi\big( (X - \sum_{h \in L}L_hX_h)^*(X - \sum_{k \in L }L_kX_k)\big)\notag \\ &\geq \,\pi\big( (X - \sum_{h \in K}L_hX_h)^*(X - \sum_{k \in K }L_kX_k)\big) \, \geq  \,0,
 \end{align}  
so the convergence of the series $\sum_gL_gX_g$ in the $\|.\|_\pi$ norm follows immediately and the norm convergence of the series $\sum_g X_g^*X_g $ follows since the inequalities above imply that there exists a finite subset $L$ such that 
$$ \eps^2 I_{B\big(H\otimes \ell_2(G)\big)} \,  \geq  \,\pi(X^*X) - \sum_{g \in L} X^*_gX_g \, \geq \, 0,$$ 
hence  for any finite subset $K$ of $G$ which contains $L$ 
$$ \eps^2 I_{B\big(H\otimes \ell_2(G)\big)} \,  \geq \pi(X^*X) - \sum_{g \in K} X^*_gX_g \, \geq \, 0,$$ so the proposition follows .
\end{proof}

\section{The Hadamard product  in C$^*_r(\ca, \a , G).$}
We may now formulate and prove the main result of this article, and we will use the projection $\pi$ from $\cc$ onto $\Psi(\ca)$  in the formulation of the theorem. We will also use the term {\em non degenerate } for the C*-algebra $\ca$  to express that span$(\ca H)$ is dense in $H.$ Since the construction below is very concrete and not an abstract existence result, we prefer to define these concretely given  items now, and then let them be part of the formulation of the theorem. 

\begin{defn} 
Let $K$ be the Hilbert space given by $ K:= \ell_2(G) \otimes H \otimes \ell_2(G),$ $(x_g)_{g \in G} $ the standard orthonomal basis for $\ell_2(G)$  and $\rho : B\big(H \otimes \ell_2(G)\big) \to B(K) $ the representation given by the amplification \newline $\rho(X) \, : = \, I_{B(\ell_2(G))} \otimes X.    $ \newline The isometry $ V : H \otimes \ell_2(G) \to K$ is given on  elementary tensors by \newline $V( \xi \otimes x_g) := x_g \otimes \xi \otimes x_g.$ \newline 
The self-adjoint unitary $F$ on $K$ is given on elementary tensors by \newline $F(x_g \otimes \xi \otimes x_h) \, := \, x_h \otimes \xi \otimes x_g.$ 
\end{defn} 
\begin{theorem}
Let $H$ be a Hilbert space, $\ca $  a non degenerate C*-algebra on $H, $   $G$ a discrete group  and $g \to \a_g $ a homomorphism of $G$ into the the group of *-automorphisms of $\ca.$ The faithful completely positive projection $\pi : \cc \to \Psi(\ca)$ is given by  $\pi(X) \, = \, V^*\rho(X)V$ and the expression $X\star Y$ given on $\cc \times \cc$ by 
$$X \star Y := V^* \rho(X) F \rho(Y) V$$ defines a completely contractive associative product on $\cc$ which satisfies the following statements    
\begin{itemize}

\item[(i)] $\forall X,Y \in \cc : \quad (X \star Y)^*\, = \,Y^* \star X^*.$
\item[(ii)] $\forall X, Y \in \cc \, \forall A \in \ca : \quad (X \star Y)\Psi(A) \,= \, X\star (Y\Psi(A))\text{ and } \Psi( A)(X \star Y)\,= \, (\Psi(A)X)\star Y.$
\item[(iii)] $\forall X, Y \in \cc, \, \forall g \in G: \quad \pi\big(L_g^*(X \star Y)\big) \, = \, \pi(L_g^*X) \pi(L_g^*Y).$
\item[(iv)]  $ \forall X, Y \in \cc\, \forall \xi, \g \in H: \, \, |\langle (X \star Y)\xi, \g\rangle | \leq \|\pi(XX^*)^{(1/2)}\g\|\|\pi(Y^*Y)^{(1/2)}\xi\|.$  
\item[(v)]  $ \forall X, Y \in \cc:$ The sum $\sum_{g \in G}\, L_g \pi(L_g^*X) \pi(L_g^*Y)$ converges in the norm  topology to $ X \star Y.$ \item[(vi)]
 $\forall X, Y  \in \cc  \text{ there exists  a  contraction } S(X,Y) \text{ on } H \otimes \ell_2(G) \text{ s. t. } \newline  X \star Y \,= \, \pi(XX^*)^{(1/2)} S(X,Y) \pi(Y^*Y)^{(1/2)}.$ 
\item[(vii)]
 $\forall X  \in \cc: \,\, -\pi(X^*X) \leq X^* \star X \,\leq \, \pi(X^*X) .$  
  \end{itemize}
\end{theorem}

\begin{proof}
 
To see that $\pi(.) \, = \, V^*\rho(.)V $ we will set the stage with some easy remarks and some notation. Let  $\cc_0$ be the linear span of all the operators $\{L_g\Psi(A)\, : \, g \in G,\, A \in \ca\,\}$  and let $X, Y$  in $\cc_0$ be given as  finite sums $X \, = \, \sum_g L_g X_g $ and $Y \, = \, \sum_g L_gY_g, $
then we see that $$ XY \, = \, \sum_g\sum_h L_gX_gL_hY_h \, = \,   \sum_g \sum_h L_{gh} \a_{h^{-1}}(X_g)Y_h \, = \, \sum_k L_k \big( \sum_h \a_{h^{-1}}( X_{kh^{-1}})Y_h\big),$$
and 
 $$X^* \, = \, \sum_g X_g^* L_g^* \, =\, \sum_g L_g^* \a_g(X_g^*) \, = \, \sum_k L_k\a_{k^{-1}}(X_{k^{-1}}^*), $$ so $\cc_0$  is a *-subalgebra of  $\cc,$ which is the norm closure of $\cc_0.$
To see that the conditional expectation $\pi(.) $ equals $V^*\rho(.)V$ on $\cc$ we will perform the computations on an operator $X \, = \, L_g \Psi(A)$ in $\cc_0, $ and we may compute as follows,
\begin{align}
\langle V^*\rho\big(L_g\Psi(A)\big)V ( \xi \otimes x_h),(\g \otimes  x_k ) \rangle \,& = \, \langle \big(I_{\ell_2(G)} \otimes L_g \Psi(A)\big) x_h \otimes \xi \otimes x_h,  x_k \otimes \g \otimes x_k \rangle \notag \\
& = \, \d(h,k) \langle  L_g ( \a_{h^{-1}}(A) \xi \otimes x_h), (\g \otimes x_k)  \rangle \notag \\ 
& = \, \d(h,k) \langle   (\a_{h^{-1}}(A) \xi \otimes x_{gh}),( \g \otimes x_k)  \rangle \notag \\ 
& = \, \d(h,k)  \delta(g,e)\langle \a_{h^{-1}}(A) \xi,\g  \rangle \notag \\  
& = \,    \delta(g,e) \langle \Psi(A) (\xi \otimes x_{h}), (\g \otimes x_k)  \rangle \notag \\ 
& = \,  \langle  \pi\big(L_g \Psi(A)\big) (\xi \otimes x_{h}), (\g \otimes x_k)  \rangle
\end{align}

Since $\pi(.) $ and $V^*\rho(.)V$ both are continuous mappings we see that for any $X$ in $\cc$ we have $\pi(X) \, = \, V^*\rho(X)V.$
It is well known that $\pi$ is $\Psi(\ca)$ bimodular, but we need a version of this which seems a bit different so we prove that for any $A$ in $\ca$ we have 
\begin{equation}
\forall A \in \ca: \quad \rho\big(\Psi(A)\big)V \, = \, V \Psi(A), \label{module}
\end{equation}
and the proof of this follows from the following line 
$$\rho\big(\Psi(A)\big)V (\xi \otimes x_k) \, = \,( x_k \otimes \a_{k_{-1}}(A)\xi \otimes x_k) \, = \, V (\a_{k_{-1}}(A)\xi \otimes x_k)\, = \, V \Psi(A) (\xi \otimes x_k).$$

We may then begin to show the validity of all the claims. The item (i) follows directly from the definition of the bilinear operator $\star .$ 

The item (ii) 
follows  from the definition of $\star$ and the equation (\ref{module}).
 
With respect to (iii) we will first establish the result in $\cc_0,$ by showing  that for $X \, = \, L_s\Psi(A)$ and $Y\, = \, L_t\Psi(B)$ in $\cc_0 $ we have $X \star Y \, = \, \delta(s,t) L_s \Psi(AB).$ We do this via the following computations, 
\begin{align}
&\, \langle (X \star Y) (\xi \otimes x_h)), (\g \otimes x_k)\rangle \notag \\  
 = &\, \langle V^* \rho\big(L_s\Psi(A)\big)F\rho\big( L_t\Psi(B)\big)V (\xi \otimes x_h)), (\g \otimes x_k)\rangle \notag \\ 
  = & \, \langle F \rho\big( L_t\Psi(B)\big) (x_h \otimes \xi \otimes x_h),\rho\big(\Psi(A^*)L_{s^{-1}}\big) (x_k \otimes \g \otimes x_k)\rangle \notag \\
 = & \, \langle F  (x_h \otimes  \a_{h^{-1}}(B)\xi \otimes x_{th}), (x_k \otimes \a_{k^{-1}s}(A^*)\g \otimes x_{s^{-1}k})\rangle \notag \\ 
 = & \, \langle  (x_{th} \otimes  \a_{h^{-1}}(B)\xi \otimes x_{h}), (x_k \otimes \a_{k^{-1}s}(A^*)\g \otimes x_{s^{-1}k})\rangle \notag \\  
  = & \,\d(th, k) \d(h, s^{-1}k)\langle  a_{h^{-1}}(B)\xi ,  \a_{k^{-1}s}(A^*)\g \rangle \notag \\ 
 = & \,\d(th,k)\d(h, s^{-1}k)\langle  \a_{h^{-1}}(AB)\xi ,  \g \rangle \notag \\ 
  = & \,\d(s,t) \d(sh,k) \langle a_{h^{-1}}(AB)\xi ,  \g \rangle \notag \\ 
   = & \,\langle \d(s,t)\big(   L_s\Psi(AB)\big)(\xi \otimes x_h),  (\g \otimes x_k)  \rangle . 
\end{align}
Hence for any pair $X, Y$ in $\cc_0$ and any $g$ in $G $ we have \begin{equation}
\pi\big(L_g^*(X \star Y)\big) \, = \, \pi(L_g^*X)\pi(L_g^*Y),
\label{Hada} \end{equation}
and by continuity we conclude that (iii) is valid. 

 The inequality (iv) implies Livshits' inequality \cite{Li}, which was fundamental  in establishing  estimates  on norms of commutators in \cite{C0}. When, in \cite{C1} section 3,  we first realized the validity of (iv) - for Schur products - we were unaware of the fact that Horn and Johnson in \cite{HJ} proof of Theorem 5.5.3 presents the same calculations as you will see below in a slightly different form.  
Now let $\xi, \g$ be in $H \otimes \ell_2(G),$ then  with the aid of equation (\ref{Hada}) and  two different versions  the Cauchy-Schwarz inequality, we get for any pair $X, \, Y$ in $\cc_0$  
\begin{align}  \notag
|\langle (X\star Y)\xi, \g \rangle | & =| \sum_g \langle \pi(L_g^*Y) \xi , \pi(X^*L_g)L_g^* \g \rangle | \\ \notag &\leq \sum_g \|\pi(L_g^*Y)\xi\| \|\pi(X^*L_g)L_g^*\g\| \\ \notag &\leq \big(\sum_g \|\pi(L_g^*Y) \xi\|^2\big)^{(1/2)} \big(\sum_g \|\pi(X^*L_g)L_g^* \g\|^2\big)^{(1/2)} \\\notag
& = \big(\sum_g \| Y_g \xi\|^2\big)^{(1/2)} \big(\sum_g \|X_g^*L_g^* \g\|^2\big)^{(1/2)} \\ \notag
& = \big( \langle \sum_g  Y_g^*Y_g \xi, \xi \rangle \big)^{(1/2)} \big( \langle \sum_g L_gX_gX_g^*L_g^* \g, \g \rangle \big)^{(1/2)} \\ \notag
& = \big( \langle \pi(Y^*Y) \xi, \xi \rangle \big)^{(1/2)} \big( \langle \pi(XX^*) \g, \g \rangle \big)^{(1/2)}\\ \label{BasIneq}
& = \|\pi(Y^*Y)^{(1/2)} \xi\| \|\pi(XX^*)^{(1/2)} \g\| \\ \notag
& \leq \|\pi(XX^*)^{(1/2)}\|\|\pi(Y^*Y)^{(1/2)}\|  | \|\xi\| \|\g\|
\end{align}

We know that the product $\star$ is continuous so we may extend the inequalities above to all of $\cc,$ and the statement (iv) follows.

To prove item (v) we let $X = \sum_gL_gX_g$ and $ Y = \sum_g L_gY_g$  be a pair of operators in $\cc$ and  $\eps $ a positive real number. Then $X^* = \sum_g L_g^*( L_gX_g^*L_g^*)$ and by Proposition \ref{normconv} there exists a finite subset $K$ of $G$ such that 
$$\| \sum_{ g \notin K} L_gX_gX_g^*L_g^* \| \leq \eps^2 \text{ and } \|\sum_{g \notin K} Y_g^*Y_g \|\leq \eps^2 .$$ 
We may  define $\hat X := \sum_{g \notin K } L_gX_g ,$ $\hat Y := \sum_{g \notin K } L_gX_g $ and $\widehat{X \star Y} : =  \hat X \star \hat Y = \sum_{g \notin K } L_gX_g Y_g.$ Let  $\xi, \g$ be a pair of vectors in $H, $  then the inequalities from (\ref{BasIneq}) may be reused on the product $\hat X \star \hat Y$ to obtain $$|\langle (\hat X \star \hat Y) \xi , \g \rangle | \leq \big( \langle \sum_{g \notin K}  Y_g^*Y_g \xi, \xi \rangle \big)^{(1/2)} \big( \langle \sum_{g \notin K} L_gX_gX_g^*L_g^* \g, \g \rangle \big)^{(1/2)} \leq \eps^2 \|\xi\|\|\g \|. $$  so $\|\widehat{X \star Y }\| \leq \eps^2$ and the series $\sum_g L_gX_gY_g$ is norm convergent with sum $X\star Y,$ and statement (v) follows.

With respect to item (vi), we remark that for $X, Y $ in $\cc$ we have by definition  $ X \star Y \,  = \, V^* \rho(X) F \rho(Y) V.$ The operator $\rho(Y)V$  has a polar decomposition such that there exist a partial isometry $W_Y$ in $B\big(H \otimes \ell_2(G),K\big) $ for which $ \rho(Y)V = W_Y \pi(Y^*Y)^{(1/2)}.  $ In a similar way we may find a partial isometry $ W_X $ in $B\big(K,H \otimes \ell_2(G)\big) $ such that $\pi(XX^*)^{(1/2)} W_X  = V^*\rho(X),$ and for the contraction $ S(X,Y) $ in $B(H \otimes \ell_2(G))$ which is defined by 
$S(X,Y) \, := \, W_XFW_Y$ we then have
\begin{align*} X \star Y \,&= \,V^*\rho(X) F \rho(Y)V \\ & = \, \pi(XX^*)^{(1/2)} W_X F W_Y  \pi(Y^*Y)^{(1/2)} \, = \, \pi(XX^*)^{(1/2)}   S(X,Y) \pi(Y^*Y)^{(1/2)},
\end{align*} and item (vi) follows. 

The statement (vii) follows from (vi), since the operator $S(X^*,X)$ will be a self-adjoint contraction, and the theorem follows.
\end{proof}

 \section{ The block Schur product as a  Hadamard  product} 
 
The block Schur product of two  matrices $A = (a_{ij} )$ and $B = (b_{ij} )$  with entries in an operator algebra $\ca$ is defined as the block matrix $A \square B:= (a_{ij}b_{ij}). $ We will take a look at a particular case of a covariant representation of a crossed product, which turns out to give us the block Schur product on $M_n\big(B(H)\big).$ So let $n $ be a natural number, $H$ a Hilbert space and 
 $C_n := \bz / n\bz$  denote the cyclic group of order $n.$   Define  $L : = H \otimes \ell_2(C_n) $ and let the C*-algebra $\ca $ on $L$ to be given as  $B(H) \otimes \ell_\infty(C_n)\, = \, \ell_\infty\big(C_n, B(H)\big),$  which we will think of  as the diagonal operators in $M_{C_n}\big(B(H)\big).$ The action of $C_n$ on $\ca$ is given as $\a_f(A)(g) := A(g-f)$ and a unitary representation of $C_n$ on $L$ implementing this group of automorphisms is given by $(L_f\Xi)(g) := \Xi(g-f).$ The C*-algebraic crossed product $\cc$ generated by this covariant representation $(\ca,  L_g, L)$ of the C*-dynamical system $(\ca, \a_g, C_n) $ is nothing but \newline $M_{C_n}\big(B(H)\big) = B(H) \otimes M_{C_n}(\bc) $ on $H \otimes \ell_2(C_n) .$ It is then well known that $M_n\big(B(H)\big)$ will be isomorphic to C$^*_r(\ca, \a_g, C_n).$  
 
 The projection $\pi$ of $\cc$  onto $\ca$ is in this case the diagonal projection, i.e. for $X=(x_{ij} )$ in $M_{C_n}\big(B(K)\big)$ we have  $\pi(X)_{kl} = \d(k,l)x_{kk}.$  By the theory an $X$ in $\cc$ may be written as $$ X = \sum_{g \in C_n} L_gX_g \text{ with } X_g \text{ a diagonal matrix }, $$ 
 and the Hadamard product for $X = \sum_g L_gX_g$ and $Y = \sum_g L_gY_g $ is given as \begin{equation} \label{H=S} 
  X \star Y = \sum_{g \in C_n}  L_g X_g Y_g.
  \end{equation} 
Since the operators $L_g$ have matrices which have the element $I$ as entries in the $g'$ th diagonal and zero entries  elsewhere, the equation (\ref{H=S}) implies that this Hadamard product is nothing but the block Schur product.

\end{document}